\documentclass[12pt]{amsart}

\usepackage{color}
\usepackage{amsmath, amsthm, amssymb}
\usepackage{amsfonts}
\usepackage[ansinew]{inputenc}
\usepackage[dvips]{epsfig}
\usepackage{graphicx}
\usepackage{graphics}
\usepackage[english]{babel}
\usepackage{hyperref}
\theoremstyle{plain}
\newtheorem{thm}{Theorem}[section]

\newtheorem{lem}[thm]{Lemma}
\newtheorem{prop}[thm]{Proposition}

\theoremstyle{definition}

\theoremstyle{remark}
\newtheorem{rem}[thm]{Remark}

\numberwithin{equation}{section}

\newcommand{\average}{{\mathchoice {\kern1ex\vcenter{\hrule height.4pt
width 6pt depth0pt} \kern-9.7pt} {\kern1ex\vcenter{\hrule
height.4pt width 4.3pt depth0pt} \kern-7pt} {} {} }}

\begin{document}

\title[Stable solutions in domains of double revolution]
{Regularity of stable solutions up to dimension 7 in domains of double revolution}

\author{Xavier Cabr\'e}

\address{ICREA and Universitat Polit\`ecnica de Catalunya,
Departament de Matem\`{a}tica  Aplicada I, Diagonal 647, 08028 Barcelona, Spain}

\email{xavier.cabre@upc.edu}

\thanks{Both authors were supported by MTM2008-06349-C03-01,
MTM2011-27739-C04-01 (Spain) and 2009SGR345 (Catalunya).}

\author{Xavier Ros-Oton}

\address{Universitat Polit\`ecnica de Catalunya, Departament de Matem\`{a}tica
Aplicada I, Diagonal 647, 08028 Barcelona, Spain}
\email{xavier.ros.oton@upc.edu}


\begin{abstract} We consider the class of semi-stable positive solutions to
semilinear equations $-\Delta u=f(u)$ in a bounded domain $\Omega\subset\mathbb R^n$
of double revolution, that is, a domain invariant under rotations of the first $m$
variables and of the last $n-m$ variables. We assume $2\leq m\leq n-2$.
When the domain is convex, we establish a priori $L^p$ and $H^1_0$ bounds for each dimension $n$,
with $p=\infty$ when $n\leq7$. These estimates lead to the boundedness of the extremal solution
of $-\Delta u=\lambda f(u)$ in every convex domain of double revolution when $n\leq7$.
The boundedness of extremal solutions is known when $n\leq3$ for any domain $\Omega$,
in dimension $n=4$ when the domain is convex, and in dimensions $5\leq n\leq9$ in the radial case.
Except for the radial case, our result is the first partial answer valid for all nonlinearities $f$ 
in dimensions $5\leq n\leq 9$.
\end{abstract}

\maketitle

\section{Introduction and results} \label{intro}

Let $\Omega \subset\mathbb R^n$ be a smooth and bounded domain, and consider the problem
\begin{equation}\label{pb}\left\{ \begin{array}{rcll} -\Delta u &=&\lambda f(u)&\textrm{in }\Omega \\
u&>&0&\textrm{in }\Omega\\
u&=&0&\textrm{on }\partial\Omega,\end{array}\right.\end{equation}
where $\lambda$ is a positive parameter and the nonlinearity $f:[0,\infty)\longrightarrow\mathbb R$
satisfies
\begin{equation}\label{condicions}f\textrm{ is } C^{1},
\textrm{ nondecreasing},\ f(0)>0,\ \textrm{and}\ \lim_{\tau\rightarrow\infty}
\frac{f(\tau)}{\tau}=\infty.\end{equation}

It is well known (see the excellent monograph \cite{D} and references therein) that
there exists an extremal parameter $\lambda^*\in(0,\infty)$ such that if $0<\lambda<\lambda^*$
then problem \eqref{pb} admits a minimal classical solution $u_\lambda$,
while for $\lambda>\lambda^*$ it has no solution, even in the weak sense.
Here, minimal means smallest. Moreover, the set $\{u_\lambda:0<\lambda<\lambda^*\}$
is increasing in $\lambda$, and its pointwise limit
$u^*=\lim_{\lambda\rightarrow \lambda^*}u_\lambda$ is a weak solution of problem \eqref{pb}
with $\lambda=\lambda^*$. It is called the extremal solution of \eqref{pb}.

When $f(u)=e^u$, it is well known that $u\in L^\infty(\Omega)$ if $n\leq9$, while
$u^*(x)=\log\frac{1}{|x|^2}$ if $n\geq10$ and $\Omega=B_1$. An analogous result holds
for $f(u)=(1+u)^p$, $p>1$. In the nineties H. Brezis and J.L. V\'azquez \cite{BV} raised the question
of determining the regularity of $u^*$, depending on the dimension $n$, for general convex
nonlinearities satisfying \eqref{condicions}. The first general results were proved by
G. Nedev \cite{N,N2} ---see \cite{CS} for the statement and proofs of the results of \cite{N2}.

\begin{thm}[\cite{N},\cite{N2}]\label{thnedev} Let $\Omega$ be a smooth bounded domain,
$f$ be a function satisfying \eqref{condicions} which in addition is convex, and $u^*$ be the extremal solution
of \eqref{pb}. \begin{itemize}
\item[i)] If $n\leq3$, then $u^*\in L^\infty(\Omega)$.
\item[ii)] If $n\geq4$, then $u^*\in L^p(\Omega)$ for every $p<\frac{n}{n-4}$.
\item[iii)] Assume either that $n\leq5$ or that $\Omega$ is strictly convex. Then $u^*\in H^1_0(\Omega)$.
\end{itemize}
\end{thm}
In 2006, the first author and A. Capella \cite{CC} studied the radial case.
Their result establishes optimal $L^\infty$ and $L^p$ regularity results in every dimension for general $f$.
\begin{thm}[\cite{CC}]\label{thradial} Let $\Omega=B_1$ be the unit ball in $\mathbb R^n$,
$f$ be a function satisfying \eqref{condicions}, and $u^*$ be the extremal solution of \eqref{pb}.
\begin{itemize}
\item[i)] If $n\leq9$, then $u^*\in L^\infty(\Omega)$.
\item[ii)] If $n\geq10$, then $u^*\in L^p(\Omega)$ for every $p<p_n$, where
\begin{equation}\label{radexp}p_n=2+\frac{4}{\frac{n}{2+\sqrt{n-1}}-2}.\end{equation}
\item[iii)] For every dimension $n$,  $u^*\in H^3(\Omega)$.
\end{itemize}
\end{thm}
The best known result was established in 2010 by the first author \cite{C4} and establishes
the boundedness of $u^*$ in convex domains in dimension $n=4$. Related ideas recently allowed
the first author
and M. Sanch\'on \cite{CS} to improve Nedev's $L^p$ estimates of Theorem \ref{thnedev}
when $n\geq5$:
\begin{thm}[\cite{C4},\cite{CS}]\label{thcabre} Let $\Omega\subset \mathbb R^n$ be a convex,
smooth and bounded domain, $f$ be a function satisfying \eqref{condicions}, and $u^*$ be the
extremal solution of \eqref{pb}.\begin{itemize}
\item[i)] If $n\leq4$, then $u^*\in L^\infty(\Omega)$.
\item[ii)] If $n\geq5$, then $u^*\in L^p(\Omega)$ for every $p<\frac{2n}{n-4}=2+\frac{4}{\frac n2-2}$.
\end{itemize}
\end{thm}
The boundedness of extremal solutions remains an open question in dimensions $5\leq n\leq 9$,
even in the case of convex domains and convex nonlinearities.

The aim of this paper is to study the regularity of the extremal solution $u^*$ of \eqref{pb}
in a class of domains that we call of double revolution. The class contains domains much
more general than balls, but is much simpler than general convex domains.
In this class of domains our main result establishes the boundedness of the extremal solution $u^*$ in dimensions 
$n\leq7$, whenever $\Omega$ is convex.
An interesting point of our work is that it has led us to a new Sobolev and isoperimetric  inequality 
(Proposition \ref{sobx^A} below) with a
monomial weight or density. In a future paper \cite{CR}, we treat a more general version of these Sobolev 
and isoperimetric inequalities with
densities (see Remark \ref{remsob} below) for which we can compute best constants, as well as extremal sets 
and functions.
They are in the spirit of recent works on manifolds with a density; see F. Morgan's survey \cite{Mo} for 
more information.

Let $n\geq4$ and
\begin{equation}\label{double}\mathbb R^n=\mathbb R^m\times \mathbb R^k\ \textrm{ with }
\ n=m+k,\ m\geq2,\ \textrm{ and }\ k\geq 2.\end{equation}
For each $x\in\mathbb R^n$ we define the variables
\[
\left\{ \begin{array}{rcl} s&=&\sqrt{x_1^2+\cdots+x_m^2}
\vspace{2mm}  \\
t&=&\sqrt{x_{m+1}^2+\cdots+x_n^2}.\end{array}\right.
\]
We say that a domain $\Omega\subset \mathbb{R}^n$ is a \emph{domain of double revolution}
if it is invariant under rotations of the first $m$ variables and also under rotations of
the last $k$ variables.
Equivalently, $\Omega$ is of the form $\Omega=\{x\in\mathbb R^n:\ (s,t)\in \Omega_2\}$ where $\Omega_2$
is a domain in $\mathbb R^2$ symmetric with respect to the two coordinate axes. In fact,
$\Omega_2=\{(y_1,y_2)\in\mathbb R^2:\ x=(x_1=y_1,x_2=0,...,x_m=0,x_{m+1}=y_2,x_{m+2}=0,...,x_n=0)\in\Omega\}$ is
the intersection of $\Omega$ with the $(x_1,x_{m+1})$-plane.
Note that $\Omega_2$ is smooth if and only if $\Omega$ is smooth.
Let us call $\widetilde\Omega$ the intersection of $\Omega_2$ with the positive quadrant of
$\mathbb R^2$, i.e.,
\begin{equation}\label{omtil}
\begin{split}
\widetilde\Omega=\big\{(s,t)\in \mathbb R^2 & :  s>0,t>0,\text{ and }\\
& \hspace{-1cm} (x_1=s,x_2=0,...,x_m=0,x_{m+1}=t,x_{m+2}=0,...,x_n=0)\in\Omega \big\}.
\end{split}
\end{equation}
Since $\{s=0\}$ and $\{t=0\}$ have zero measure in $\mathbb R^2$, we have that
\[\int_\Omega v\ dx=c_{m,k}\int_{\widetilde\Omega}v(s,t)s^{m-1}t^{k-1}dsdt\]
for every $L^1(\Omega)$ function $v=v(x)$ which depends only on the radial variables $s$ and~$t$.
Here, $c_{m,k}$ is a positive constant depending only on $m$ and $k$.

In the previous theorems, the regularity of $u^*$ is proved using its semi-stability. More precisely, the
minimal solutions $u_\lambda$ of \eqref{pb} turn out to be semi-stable solutions.
A solution is semi-stable if the second variation of energy at the solution is nonnegative; see \eqref{stcond} below.
We will prove that any semi-stable classical solution $u$ of \eqref{pb}, and more generally of
\eqref{pbl} below, depends only on $s$ and $t$, and hence we can identify it with a function
$u=u(s,t)$ defined in $(\mathbb R_+)^2=(0,\infty)^2$ which satisfies the equation
\begin{equation}\label{eqst} u_{ss}+u_{tt}+\frac{m-1}{s}u_s+\frac{k-1}{t}u_t+f(u)=0\ \ 
\textrm{ for } (s,t)\in\widetilde\Omega.\end{equation}
Moreover, in the case of convex domains we will also have $u_s\leq0$ and $u_t\leq0$ (for $s>0$, $t>0$) and hence,
$u(0)=\|u\|_{L^\infty}$ (see Remark \ref{symm}).

The following is our main result. We prove that, in convex domains of double revolution,
the extremal solution $u^*$ is bounded when $n\leq7$, and it belongs to $H^1_0$ and certain
$L^p$ spaces when $n\geq8$. We also prove that in dimension $n=4$ the convexity of the domain
is not required for the boundedness of $u^*$ (in \cite{C4}, convexity of $\Omega$ was a requirement 
in general domains of $\mathbb R^4$).
\begin{thm}\label{th1} Assume \eqref{double}. Let $\Omega\subset \mathbb{R}^n$ be a smooth and
bounded domain of double revolution, $f$ be a function satisfying \eqref{condicions}, and
$u^*$ be the extremal solution of \eqref{pb}.
\begin{itemize}
\item[a)] Assume either that $n=4$ or that $n\leq7$ and $\Omega$ is convex. Then,
$u^*\in L^{\infty}(\Omega)$.
\item[b)] If $n\geq8$ and $\Omega$ is convex, then $u^*\in L^p(\Omega)$ for all $p<p_{m,k}$,
where
\begin{equation}\label{pmk}p_{m,k}=2+\frac{4}{\frac{m}{2+\sqrt{m-1}}+\frac{k}{2+\sqrt{k-1}}-2}.
\end{equation}
\item[c)] Assume either that $n\leq6$ or that $\Omega$ is convex. Then, $u^*\in H^1_0(\Omega)$.
\end{itemize}
\end{thm}

\begin{rem}\label{qmkconvex} Let $q_{m,k}=\frac{m}{2+\sqrt{m-1}}+\frac{k}{2+\sqrt{k-1}}$.
Since $q(x):=\frac{x}{2+\sqrt{x-1}}$ is a concave function in $[2,\infty)$, we have 
$q'(x)-q'(n-x)\geq0$ in $[2,\frac n2]$, and thus $q(x)+q(n-x)$ is nondecreasing in $[2,\frac n2]$. 
Hence, $q_{2,n-2}\leq q_{m,k}
\leq q_{\frac n2,\frac n2}$, and therefore $p_{\frac n2,\frac n2}\leq p_{m,k}\leq p_{2,n-2}$.
Thus, asymptotically as $n\rightarrow\infty$,
\[2+\frac{2\sqrt2}{\sqrt n}\simeq p_{\frac{n}{2},\frac{n}{2}}\leq p_{m,k}\leq p_{2,n-2}\simeq
2+\frac{4}{\sqrt n}.\]
Instead, in a general convex domain, $L^p$ estimates are only known for $p\simeq 2+\frac 8 n$
(see Theorem \ref{thcabre} ii above), while in the radial case one has $L^p$ estimates for
$p\simeq 2+\frac{4}{\sqrt n}$ (see Theorem \ref{thradial} ii).
\end{rem}

The proofs of the results in \cite{N,N2,CC,C4,CS} use the semi-stability of the extremal
solution $u^*$. In fact, one first proves estimates for any regular semi-stable solution $u$ of
\begin{equation}\label{pbl}\left\{ \begin{array}{rcll} -\Delta u &=&f(u)&\textrm{in }\Omega \\
u&=&0&\textrm{on }\partial\Omega,\end{array}\right.\end{equation}
then one applies these estimates to the minimal solutions $u_\lambda$ (which are semi-stable), and finally by
monotone convergence the estimates also hold for the extremal solution $u^*$.

Recall that a classical solution $u$ of \eqref{pbl} is said to be \emph{semi-stable} if
the second variation of energy at $u$ is nonnegative, i.e., if
 \begin{equation}\label{stcond} Q_{u}(\xi)=\int_\Omega\left\{|\nabla\xi|^2-f'(u)\xi^2\right\}dx \geq0
\end{equation} for all $\xi\in C^1_0(\overline\Omega)$. For instance, every local minimizer of 
the energy is a semi-stable solution.

The proof of the estimates in \cite{CC,C4,CS} was inspired by the proof of Simons theorem on the
nonexistence of singular minimal cones in $\mathbb R^n$ for $n \leq 7$ (see \cite{CC2} for more details). 
The key idea is to take
$\xi=|\nabla u|\eta$ (or $\xi=u_r\eta$ in the radial case) and compute $Q_{u}(|\nabla u|\eta)$
in the semi-stability
property satisfied by $u$. In this way the expression of $Q_u$ in terms of $\eta$ turns out not to depend on
$f$ and, thanks to this, a clever choice of the test function $\eta$ leads to $L^p$ and $L^\infty$ bounds
depending on the dimension $n$ but valid for all nonlinearities $f$.

In this paper we will proceed in a similar way, proving first results for general positive semi-stable
solutions of \eqref{pbl} and then applying them to $u_\lambda$ to deduce estimates for $u^*$.
We will take $\xi=u_s\eta$ and $\xi=u_t\eta$ separately instead of $\xi=|\nabla u|\eta$, 
and this will lead to bounds for
\begin{equation}\label{estsemi}\int_\Omega u_s^2 s^{-2\alpha-2}dx\qquad\textrm{and}\qquad
\int_\Omega u_t^2 t^{-2\beta-2}dx\end{equation}
for any $\alpha<\sqrt{m-1}$ and $\beta<\sqrt{k-1}$.

When the domain $\Omega$ is convex,
we will have the additional information $\|u\|_{L^\infty}=u(0)$, $u_s\leq0$, and $u_t\leq0$,
which combined with \eqref{estsemi} will lead to $L^{\infty}$ and $L^p$ estimates for $u^*$.

Instead, when the domain $\Omega$ is not convex the maximum of $u$ may not be achieved at
the origin ---see Figure 1 for an example in which $u(0)$ will be much smaller than
$\|u\|_{L^\infty}$. Thus, in nonconvex domains we can not apply the same argument.
However, if the maximum is away from $\{s=0\}$ and $\{t=0\}$ (as in Figure 1) then
the problem is essentially two dimensional near the maximum, since $dx=c_{m,k}s^{m-1}t^{k-1}dsdt$
and both $s$ and $t$ will be positive and bounded below around the maximum. Thus, the two dimensional Sobolev inequality
will hold near the maximum. We will still have to prove some
boundary estimates, for instance estimates near the boundary points $P$ and $Q$ in Figure 1. But,
by the same reason as before, near $P$ the coordinate $s$ is positive and bonded below.
Thus, the problem near $P$ will be essentially $1+k$ dimensional, and
we assume $k=n-m\leq n-2$.
This will allow us, if $1+k\leq n-1$ are small enough, to use Nedev's \cite{N} $W^{2,p}$ estimates
to obtain boundary estimates.

\begin{figure}
\centering
\includegraphics{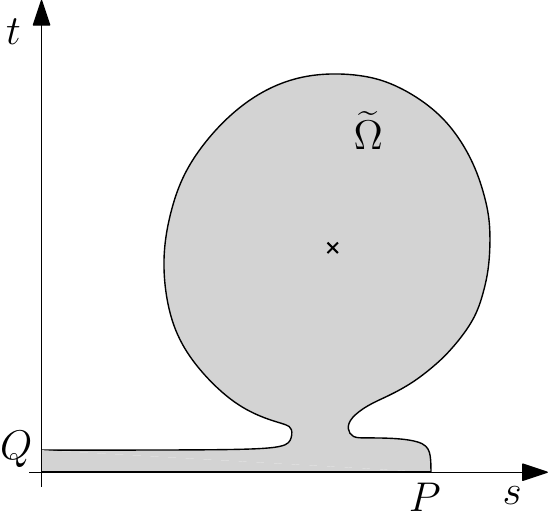}
\caption{A non-convex domain for which the maximum of $u^*$ will not be $u^*(0)$}
\end{figure}

Our result for general positive semi-stable solutions of \eqref{pbl} reads as follows. It states
global estimates controlled in terms of boundary estimates.

\begin{prop} \label{prop1} Assume \eqref{double}. Let $\Omega\subset \mathbb{R}^n$ be a
smooth and bounded domain of double revolution, $f$ be any $C^{1}$ function, and
$u$ be a positive bounded semi-stable solution of \eqref{pbl}.

Let $\delta$ be any positive real number, and define \[\Omega_{\delta}=\{x\in\Omega:
{\rm dist}(x,\partial\Omega)<\delta\}.\]
Then, for some constant $C$ depending only on $\Omega$, $\delta$, $n$, and also $p$ in
part b) below, one has:
\begin{itemize}
\item[a)] If $n\leq7$ and $\Omega$ is convex, then $\|u\|_{L^{\infty}(\Omega)}\leq
C\left(\|u\|_{L^{\infty}(\Omega_{\delta})}+\|f(u)\|_{L^{\infty}(\Omega_{\delta})}\right)$.
\item[b)] If $n\geq8$ and $\Omega$ is convex, then $\|u\|_{L^p(\Omega)}\leq
C\left(\|u\|_{L^{\infty}(\Omega_{\delta})}+\|f(u)\|_{L^{\infty}(\Omega_{\delta})}\right)$
for each $p<p_{m,k}$, where $p_{m,k}$ is given by \eqref{pmk}.
\item[c)] For all $n\geq4$, $\|u\|_{H^1_0(\Omega)}\leq C\|u\|_{H^1(\Omega_{\delta})}$.
\end{itemize}
\end{prop}

To prove part b) of Proposition \ref{prop1} we will need a new weighted Sobolev
inequality in $(\mathbb R_+)^2=\{(\sigma,\tau)\in\mathbb R^2:\sigma>0, \tau>0\}$. 
We will use this inequality in the $(\sigma,\tau)$-plane defined after the change of variables
\[\sigma=s^{2+\alpha},\ \ \tau=t^{2+\beta},\] where $\alpha$ and $\beta$ are the 
exponents in \eqref{estsemi}.
It states the following.

\begin{prop} \label{sobx^A} Let $a>-1$ and $b>-1$ be real numbers, being positive at least
one of them, and let \[D=2+a+b.\]
Let $u$ be a nonnegative Lipschitz function with compact support in $\mathbb R^2$ such that $u\in C^1(\{u>0\})$,
\[u_\sigma\leq0\ \textrm{ and }\ u_\tau\leq0\ \textrm{ in }\ (\mathbb R_+)^2,\] with strict
inequalities whenever $u>0$.
Then, for each $1\leq q<D$ there exists a constant $C$, depending only on $a,\ b$, and $q$,
such that \begin{equation}\label{ineqSobol}\left(\int_{(\mathbb R_+)^2}\sigma^a\tau^b
|u|^{q^*}d\sigma d\tau\right)^{1/q^*}\leq C\left(\int_{(\mathbb R_+)^2} \sigma^a\tau^b
|\nabla u|^qd\sigma d\tau\right)^{1/q},
\end{equation} where $q^*=\frac{Dq}{D-q}$.
\end{prop}

\begin{rem}\label{remsob} When $a$ and $b$ are nonnegative integers, inequality \eqref{ineqSobol}
is a direct consequence of the classical Sobolev inequality in $\mathbb R^D$.
Namely, define in $\mathbb R^D=\mathbb R^{a+1}\times \mathbb R^{b+1}$ the radial variables 
$\sigma=|(x_1,\ldots,x_{a+1})|$ and $\tau=|(x_{a+2},\ldots,x_D)|$. 
Then, for functions $u$ defined in $\mathbb R^D$ depending
only on the variables $\sigma$ and~$\tau$, write the integrals appearing in the classical Sobolev
inequality in $\mathbb R^D$ in terms of $\sigma$ and~$\tau$. 
Since $dx=c_{a,b}\sigma^a\tau^bd\sigma d\tau$, the obtained
inequality is precisely the one given in Proposition~\ref{sobx^A}.

Thus, the previous proposition extends the classical Sobolev inequality to the case of
non-integer exponents $a$ and $b$. In another article, \cite{CR}, we prove inequality
\eqref{ineqSobol} with $(\mathbb R_+)^2$ replaced by $(\mathbb R_+)^d$ and with $\sigma^a\tau^b$
replaced by the monomial weight 
\[x^A:=x_1^{A_1}\cdots x_d^{A_d},\]
where $A_1,...,A_d$ are nonnegative real numbers.
We also prove a related
isoperimetric inequality with best constant, a weighted Morrey's inequality, and we determine
extremal sets and functions for some of these inequalities.
\end{rem}

In section 4 we establish the weighted Sobolev inequality of Proposition~\ref{sobx^A} 
as a consequence of a
new weighted isoperimetric inequality. Our proof is simple but does not give the best 
constant (in contrast with the more involved 
proof that we will give in \cite{CR} giving the best constant).
When $a$ and $b$ belong to $(0,q-1)$ ---i.e., $(0,1)$ when $q=2$, as in our application) inequality 
\eqref{ineqSobol} also follows from a result of P. Hajlasz \cite{H} 
in a very general framework of weights or measures. His result does not give the 
best constant and, besides, its constant depends on the support of the function. 

We will need to use the proposition for some exponents $a$ and $b$ in $(-1,0)$ 
---this happens for instance when $m=2$ or $m=3$. In this case the assumption $u_\sigma\leq0$, 
$u_\tau\leq0$ is crucial
for the inequality to hold with the optimal exponent $q^*$. Without this assumption, 
a Sobolev inequality is still true but with a smaller exponent than $q^*$ (this also follows 
from the results in \cite{H}). For $a>q-1$ the weight is no longer in the Muckenhoupt 
class $A_q$ and the results in \cite{H} do not apply.

The paper is organized as follows. In section 2 we prove the estimates of Proposition \ref{prop1}.
Section 3 deals with the regularity of the extremal solution of \eqref{pb}. Finally, in section 4
we prove the weighted Sobolev inequality of Proposition \ref{sobx^A}.

\section{Proof of Proposition \ref{prop1}} \label{est}

We start with a remark on the symmetry and monotonicity properties of solutions to \eqref{pbl},
as well as on the regularity of the functions $u_s$ and $u_t$.

\begin{rem}\label{symm} Note that when the domain is of double revolution, any bounded
semi-stable solution $u$ of \eqref{pbl} will depend only on the variables $s$ and $t$.
To prove this, define $v=x_iu_{x_j}-x_ju_{x_i}$, with $i\neq j$. Note that $u$ will will
depend only on $s$ and $t$ if and only if $v\equiv0$ for each $i,j\in\{1,...,m\}$ and for
each $i,j\in\{m+1,...,n\}$.

We first see that, for such indexes $i$ and $j$, $v$ is a solution of the linearized equation of \eqref{pbl}:
\begin{eqnarray*}\Delta v&=& \Delta(x_iu_{x_j}-x_ju_{x_i})\\
&=&x_i\Delta u_{x_j}+2\nabla x_i\cdot \nabla u_{x_j}-x_j\Delta u_{x_i}-2\nabla x_j\cdot \nabla u_{x_i}\\
&=&x_i(\Delta u)_{x_j}-x_j(\Delta u)_{x_i}\\
&=&-f'(u)\{x_i u_{x_j}-x_ju_{x_i}\}\\
&=&-f'(u)v.
\end{eqnarray*}
Note that $v$ is a tangential derivative of $u$ along $\partial\Omega$ since $\Omega$ is a domain of double revolution.
Therefore, since $u=0$ on $\partial \Omega$ then $v=0$ on $\partial \Omega$.
Thus, multiplying the equation by $v$ and integrating by parts, we obtain
\[\int_{\Omega} \{|\nabla v|^2-f'(u)v^2\}dx=0.\]
But since $u$ is semi-stable, the first Dirichlet eigenvalue $\lambda_1(\Delta+f'(u);\Omega)\geq0$.

If $\lambda_1(\Delta+f'(u);\Omega)>0$, the previous inequality leads to $v\equiv0$.

If $\lambda_1(\Delta+f'(u);\Omega)=0$, then we must have $v=K\phi_1$, where $K$ is a
constant and $\phi_1$ is the first Dirichlet eigenfunction of $\Delta+f'(u)$,
which we may take to be positive in $\Omega$.
But since $v$ is the derivative of $u$ along the vector field $\partial_t=x_i\partial_{x_j}-
x_j\partial_{x_i}$, and its integral curves are closed, $v$ can not have constant sign. Thus,
$K=0$, that is, $v\equiv0$.

Hence, we have seen that any classical semi-stable solution $u$ of \eqref{pbl} depends only on
the variables $s$ and $t$. Moreover, by the classical result of Gidas-Ni-Nirenberg \cite{GNN},
when $\Omega$ is even and convex with respect each coordinate and $u$ is a positive solution, 
we have $u_{x_i}\leq 0$ when $x_i>0$,
for $i=1,...,n$. In particular, when $\Omega$ is a convex domain of double revolution,
we have that $u_s<0$ and $u_t<0$ for $s>0$, $t>0$, $(s,t)\in\tilde\Omega$.
In particular, \[\|u\|_{L^{\infty}(\Omega)}=u(0).\]

On the other hand, by standard elliptic regularity for \eqref{pbl} and its linearization, 
every bounded solution $u$ of \eqref{pbl}
satisfies $u\in W^{3,p}(\Omega)\cap C^{2,\nu}(\overline\Omega)$ for all $p<\infty$ 
and $0<\nu<1$. In particular,
\[u_s\in H^2_{\rm loc}(\Omega\backslash\{s=0\})\qquad \textrm{and}\qquad 
u_t\in H^2_{\rm loc}(\Omega\backslash\{t=0\}),\]
since $u_s=u_{x_1}\frac{x_1}{s}+\cdots+u_{x_m}\frac{x_m}{s}$ and $u_t=u_{x_{m+1}}\frac{x_{m+1}}{t}
+\cdots+u_{x_n}\frac{x_n}{t}$.
In addition, since $u=u(s,t)$ is the restriction to the first quadrant of the $(x_1,x_{m+1})$-plane 
of an even $C^{2,\nu}$ function of $x_1$ and $x_{m+1}$, we deduce that
\begin{equation}\label{ustlip} u_s\in {\rm Lip}(\overline\Omega),\ u_t\in {\rm Lip}(\overline\Omega),
\ u_s=0\textrm{ when }s=0,\textrm{ and } u_t=0\textrm{ when }t=0.\end{equation}
We note that $u_s$ and $u_t$ do not belong to $C^1(\overline\Omega)$, neither to $H^2(\Omega)$. 
For instance, the solution of $-\Delta u=1$ in $B_1\subset \mathbb R^n$ is given by
$u=\frac{1}{2n}(1-s^2-t^2)$ and, thus, $u_s=-\frac 1n s$ is only Lipschitz in $\Omega$.
\end{rem}

Before proving Proposition \ref{prop1}, we will need two
preliminary results. The first one, Lemma \ref{lem1}, was already used in \cite{CC,C4}. 
In this paper we use it taking the function $c$ on its statement to be $u_s$ and $u_t$. 
Note that $c=u_s\in H^2_{\rm loc}(\Omega\backslash \{s=0\})$ but $u_s$ is not $H^2$ in a 
neighborhood in $\Omega$ of $\{s=0\}$.

\begin{lem}\label{lem1} Let $u$ be a bounded semi-stable solution of \eqref{pbl}, 
$V$ be an open set with $V\subset \Omega$,
and $c$ be a $H^2_{\rm loc}(V)$ function. Then,
\[ \int_{\Omega}c\{\Delta c+f'(u)c\}\eta^2dx\leq \int_{\Omega}c^2|\nabla \eta|^2dx\] for all
$\eta\in C^1(V)$ with compact support in $V$.
\end{lem}

\begin{proof} It suffices to set $\xi=c\eta$ in the semi-stability condition \eqref{stcond}
and then integrate by parts in $V$. \end{proof}

We now apply Lemma \ref{lem1} separately with $c=u_s$ and with $c=u_t$, and then we choose
appropriately the test function $\eta$ to get the following result. This estimate is the key
ingredient in the proof of Proposition \ref{prop1}.

\begin{lem} \label{lema} Assume \eqref{double}. Let $\Omega\subset \mathbb{R}^n$ be a smooth
and bounded domain of double revolution, $f$ be any $C^{1}$ function, and $u$ be a
positive bounded semi-stable solution of \eqref{pbl}. Let $\alpha$ and $\beta$ be such that
\[0\leq\alpha<\sqrt{m-1}\ \textrm{ and }\ 0\leq\beta<\sqrt{k-1}.\]
Then, for each $\delta>0$ there exists a constant $C$,
which depends only on $\Omega$, $\delta$, $n$, $\alpha$, and $\beta$, such that
\begin{equation}\label{e}\left(\int_{\Omega}\left\{u_s^2 s^{-2\alpha-2}+
u_t^2 t^{-2\beta-2}\right\}dx\right)^{1/2}\leq C\left(\|u\|_{L^{\infty}(\Omega_{\delta})}+
\|f(u)\|_{L^{\infty}(\Omega_{\delta})}\right),\end{equation}
where \[\Omega_{\delta}=\{x\in\Omega: {\rm dist}(x,\partial\Omega)<\delta\}.\]
\end{lem}

\begin{proof} We will prove only the estimate for $u_s^2 s^{-2\alpha-2}$; the other term
can be estimated similarly.

Differentiating \eqref{eqst} with respect to $s$, we obtain
\[\Delta u_s-(m-1)\frac{u_s}{s^2}+f'(u)u_s=0\qquad \textrm{in }\ \Omega\backslash\{s=0\}.\]
Hence, setting
$c=u_s$ in Lemma \ref{lem1} (recall that $c=u_s\in H^2_{\rm loc}(\Omega\backslash \{s=0\})$ 
by Remark \ref{symm}), we have that
\begin{equation}\label{ineqeta} (m-1)\int_{\Omega}u_s^2 \frac{\eta^2}{s^2}dx\leq\int_{\Omega}u_s^2
|\nabla \eta|^2dx\end{equation}
for all $\eta\in C^1(\Omega\backslash \{s=0\})$ with compact support in $\Omega\backslash\{s=0\}$.

We claim now that inequality \eqref{ineqeta} is valid for each $\eta\in C^1(\Omega)$ with compact 
support in $\Omega$.
Namely, take any such function $\eta$, and let $\zeta_\delta$ be a smooth function satisfying 
$0\leq\zeta_\delta\leq 1$, $\zeta_\delta\equiv0$ in $\{s\leq \delta\}$, $\zeta_\delta\equiv 1$ 
in $\{s\geq 2\delta\}$, and $|\nabla\zeta_\delta|\leq C/\delta$. Applying \eqref{ineqeta} 
with $\eta$ replaced by $\eta\zeta_\delta$ (which is $C^1$ and has compact support in 
$\Omega\backslash\{s=0\}$),
we obtain
\begin{equation}\label{ineqdelta}(m-1)\int_{\Omega}u_s^2 \frac{\eta^2\zeta_\delta^2}{s^2}dx
\leq\int_{\Omega}u_s^2|\nabla(\eta\zeta_\delta)|^2dx.\end{equation}
Now, we find
\begin{eqnarray*}\int_{\Omega}u_s^2|\nabla(\eta\zeta_\delta)|^2dx&=&\int_{\Omega}u_s^2
\left\{|\nabla \eta|^2\zeta_\delta^2+\eta^2|\nabla\zeta_\delta|^2+2\eta\zeta_\delta 
\nabla\eta\nabla\zeta_\delta \right\}dx\\
&\leq& \int_{\Omega}u_s^2|\nabla \eta|^2\zeta_\delta^2dx+\frac{C}{\delta^2}\int_{\{\delta\leq 
s\leq 2\delta\}\cap\Omega} u_s^2dx\\
&\leq& \int_{\Omega}u_s^2|\nabla \eta|^2\zeta_\delta^2dx+C\delta^{m-2}\|u_s\|^2_{L^{\infty}
(\{\delta\leq s\leq 2\delta)},
\end{eqnarray*}
where $C$ denote different positive constants, and we have used that $\eta$ and $|\nabla\eta|$ 
are bounded.
Since $u_s$ is continuous in $\overline\Omega$ and $u_s=0$ on $\{s=0\}$ by \eqref{ustlip}, 
we have $\|u_s\|_{L^\infty(\{s\leq2\delta\})}\rightarrow0$ as $\delta\rightarrow0$. Recall 
also that $m-2\geq0$.
Therefore, letting $\delta\rightarrow0$ in
\eqref{ineqdelta} we obtain \eqref{ineqeta}, and our claim is proved.

Moreover, by approximation by $C^1(\Omega)$ functions with compact support in $\Omega$, 
we see that \eqref{ineqeta} is valid also for each $\eta\in {\rm Lip}(\Omega)$ with compact 
support in $\Omega$.

Let us set $\eta=\eta_{\epsilon}$ in \eqref{ineqeta},
where \[\eta_{\epsilon}=\left\{\begin{array}{ll}
s^{-\alpha}\rho  & \textrm{if }\ s>\epsilon\\ \epsilon^{-\alpha}\rho  & \textrm{if }\
s\leq\epsilon \end{array} \right. \qquad\textrm{and}\qquad
\rho =\left\{\begin{array}{ll} 0 & \textrm{in } \Omega_{\delta/3}\\ 1 & \textrm{in }
\Omega\backslash\Omega_{\delta/2}, \end{array} \right.\]
and $\rho$ is a smooth function. Note that $\eta_\epsilon\in\textrm{Lip}(\Omega)$ and has 
compact support in $\Omega$.
Then, since $\alpha^2<\frac12(\alpha^2+m-1)<m-1$,
\[
|\nabla\eta_{\epsilon}|^2\leq\left\{\begin{array}{ll}\frac12(\alpha^2+m-1)s^{-2\alpha-2}
\rho ^2 &\textrm{in }(\Omega\backslash \Omega_{\delta/2})\cap \{s>\epsilon\}
\vspace{1mm}\\
\frac12(\alpha^2+m-1)s^{-2\alpha-2}\rho ^2+Cs^{-2\alpha}&\textrm{in }\Omega_{\delta/2}
\cap\{s>\epsilon\}
\vspace{1mm}\\
C\epsilon^{-2\alpha}&\textrm{in }\Omega\cap\{s\leq\epsilon\},
\end{array}\right.
\]
we deduce from \eqref{ineqeta}
\[\frac{m-1-\alpha^2}{2}\int_{\Omega\cap\{s>\epsilon\}}u_s^2s^{-2\alpha-2}\rho ^2dx\leq
C\int_{\Omega_{\delta/2}\cap\{s>\epsilon\}}u_s^2s^{-2\alpha}dx+C\epsilon^{-2\alpha}
\int_{\Omega\cap \{s\leq \epsilon\}}u_s^2dx,\] where $C$ denote different constants
depending only on the quantities appearing in the statement of the lemma.
Note that we can bound the dependence of the constants in $m$ and $k$ by a constant 
depending on $n$, since for each $n$ there is a finite number of possible $m$ and $k$.
Now, since $u_s\in L^\infty(\Omega)$, the last term is bounded by $C\|u_s\|^2_{L^\infty}\epsilon^{m-2\alpha}$.
Making
$\epsilon\rightarrow 0$ and using that
\begin{equation}\label{ineqal} 2\alpha<2\sqrt{m-1}\leq m,\end{equation}
we deduce
\[\int_{\Omega}u_s^2s^{-2\alpha-2}\rho ^2dx\leq C\int_{\Omega_{\delta/2}}u_s^2s^{-2\alpha}dx.\]
Hence, since $\rho\equiv1$ in $\Omega\backslash \Omega_{\delta/2}$,
\begin{equation}\label{joquese}
\int_{\Omega\backslash\Omega_{\delta/2}}u_s^2s^{-2\alpha-2}dx\leq C\int_{\Omega_{\delta/2}}
u_s^2s^{-2\alpha}dx\leq C\int_{\Omega_{\delta/2}}u_s^2s^{-2\alpha-2}dx.
\end{equation}
{From} this we deduce that, for another constant $C$,
\begin{equation}\label{joquesemes}
\int_{\Omega}u_s^2s^{-2\alpha-2}dx\leq C\int_{\Omega_{\delta/2}}u_s^2s^{-2\alpha-2}dx.
\end{equation}

Let $0<\nu<1$ to be chosen later. On the one hand, using that $u_s\in \textrm{Lip}(\overline\Omega)$ 
and $u_s(0,t)=0$ (by \eqref{ustlip}), and that $\Omega$ is smooth, we deduce that
$|u_s(s,t)|\leq Cs^\nu\|u_s\|_{C^{0,\nu}(\overline{\Omega_{\delta/2}})}$ in $\Omega_{\delta/2}
\cap \{s<\delta\}$.
Moreover, since $-\Delta u=f(u)$ in $\Omega_\delta$ and $u|_{\partial\Omega}=0$, by $W^{2,p}$ estimates we have
$\|u\|_{C^{1,\nu}(\overline{\Omega_{\delta/2}})}\leq C\left(\|u\|_{L^{\infty}(\Omega_{\delta})}+
\|f(u)\|_{L^{\infty}(\Omega_{\delta})}\right)$.
It follows that
\[\|s^{-\nu}u_s\|_{L^{\infty}(\Omega_{\delta/2}\cap \{s<\delta\})}\leq
C\left(\|u\|_{L^{\infty}(\Omega_{\delta})}+\|f(u)\|_{L^{\infty}(\Omega_{\delta})}\right).\]
Thus, also in all $\Omega_{\delta/2}$ we have
\begin{equation}\label{c1nu}\|s^{-\nu}u_s\|_{L^{\infty}(\Omega_{\delta/2})}\leq
C\left(\|u\|_{L^{\infty}(\Omega_{\delta})}+\|f(u)\|_{L^{\infty}(\Omega_{\delta})}\right).
\end{equation}

On the other hand, recalling \eqref{ineqal} and taking $\nu$ sufficiently close to 1 such that 
$m-2\alpha-2+2\nu>0$, we will have
\[\int_{\Omega_{\delta/2}}u_s^2s^{-2\alpha-2}dx\leq \|s^{-\nu}u_s\|_{L^{\infty}
(\Omega_{\delta/2})}^2\int_{\Omega_{\delta/2}}s^{-2\alpha-2+2\nu}dx\leq
C\|s^{-\nu}u_s\|_{L^{\infty}(\Omega_{\delta/2})}^2.\]
Hence, using also \eqref{joquesemes} and \eqref{c1nu},
\[\int_{\Omega}u_s^2 s^{-2\alpha-2}dx\leq
C\left(\|u\|_{L^{\infty}(\Omega_{\delta})}+\|f(u)\|_{L^{\infty}(\Omega_{\delta})}\right)^2,\]
as claimed.
\end{proof}

Using Lemma \ref{lema} we can now establish Proposition \ref{prop1}.

\vspace{3mm}

\noindent \emph{Proof of Proposition \ref{prop1}.}
Using Lemma \ref{lema} and
making the change of variables 
\[\sigma =s^{2+\alpha},\ \ \tau=t^{2+\beta}\]
in the integral in \eqref{e}, one has
\[\left\{\begin{array}{lcr} s^{m-1}ds &=&c_\alpha\sigma ^{\frac{m}{2+\alpha}-1}d\sigma \\
t^{k-1}dt &=&c_\beta\tau^{\frac{k}{2+\beta}-1}d\tau,
\end{array}\right.\]
and thus,
\begin{equation} \int_{\widetilde U}\sigma ^{\frac{m}{2+\alpha}-1}\tau^{\frac{k}{2+\beta}-1}
(u_{\sigma }^2+u_{\tau}^2)d\sigma d\tau\leq C \left(\|u\|_{L^{\infty}(\Omega_{\delta})}+
\|f(u)\|_{L^{\infty}(\Omega_{\delta})}\right)^2.\label{sbarra}\end{equation}
Here, $\widetilde U$ denotes the image of the two
dimensional domain $\widetilde\Omega$ in \eqref{omtil} after the transformation
$(s,t)\mapsto(\sigma,\tau)$.
The constant in \eqref{sbarra} depends on $\alpha$ and $\beta$. However, later we will choose 
$\alpha$ and $\beta$ depending only on $m$ and $k$ and hence the constants will be controlled 
by constants depending only on $n$ (since for each $n$ there are a finite number of integers $m$ and $k$).

a) We assume $\Omega$ to be convex. Recall that in this case $\|u\|_{L^\infty}=u(0)$;
see Remark \ref{symm}.

From \eqref{sbarra}, setting $\rho =\sqrt{\sigma ^2+\tau^2}$ and taking into account that in
$\{\tau<\sigma <2\tau\}$ we have $\frac{\rho}{2}<\sigma<\rho$ and $\frac{\rho}{3}<\tau<\rho$,
we obtain
\begin{equation}\label{c1}\int_{\widetilde U\cap\{\tau<\sigma <2\tau\}}\rho ^{\frac{m}{2+\alpha}
+\frac{k}{2+\beta}-2}(u_{\sigma }^2
+u_{\tau}^2)d\sigma d\tau\leq C \left(\|u\|_{L^{\infty}(\Omega_{\delta})}+
\|f(u)\|_{L^{\infty}(\Omega_{\delta})}\right)^2.\end{equation}

Now, for each angle $\theta$ we have
\[u(0)\leq \int_{l_{\theta}}|\nabla_{(\sigma,\tau)} u|d\rho ,\]
where $l_{\theta}$ is the segment of angle $\theta$ in the $(\sigma,\tau)$-plane from the
origin to $\partial \widetilde U$.
Integrating in $\arctan\frac12<\theta<\arctan1=\frac{\pi}{4}$,
\begin{equation}\label{c2}u(0)\leq C\int_{\arctan\frac12}^{\frac{\pi}{4}}\int_{l_{\theta}}
|\nabla_{(\sigma,\tau)} u|d\rho d\theta=C\int_{\widetilde U\cap\{\tau<\sigma<2\tau\}}
\frac{|\nabla_{(\sigma,\tau)} u|}{\rho }
d\sigma d\tau.\end{equation}
Now, applying Schwarz's inequality and taking into account \eqref{c1} and \eqref{c2},
\[u(0)\leq C\left(\|u\|_{L^{\infty}(\Omega_{\delta})}+
\|f(u)\|_{L^{\infty}(\Omega_{\delta})}\right)\left(\int_{\widetilde U\cap\{\tau<\sigma<2\tau\}}
\rho^{-\left(\frac{m}{2+\alpha}+\frac{k}{2+\beta}\right)}d\sigma d\tau\right)^{1/2}.\]
This integral is finite when \[\frac{m}{2+\alpha}+\frac{k}{2+\beta}<2.\]
Therefore, if
\begin{equation}\label{cond}\frac{m}{2+\sqrt{m-1}}+\frac{k}{2+\sqrt{k-1}}<2\end{equation}
then we can choose $\alpha<\sqrt{m-1}$ and $\beta<\sqrt{k-1}$ such that the integral is finite.
Hence, since $\|u\|_{L^{\infty}(\Omega)}=u(0)$, if condition \eqref{cond} is satisfied then
\[\|u\|_{L^{\infty}(\Omega)}\leq C\left(\|u\|_{L^{\infty}(\Omega_{\delta})}+
\|f(u)\|_{L^{\infty}(\Omega_{\delta})}\right).\]

Let \[q_{m,k}=\frac{m}{2+\sqrt{m-1}}+\frac{k}{2+\sqrt{k-1}}.\]
If $n\leq7$ then by Remark \ref{qmkconvex} we have that $q_{m,k}\leq q_{\frac{n}{2},\frac{n}{2}}
\leq q_{\frac72,\frac72}<2$ (note that
the function $q=q(x)$ in the remark is increasing in $x$).
Instead, if $n\geq8$ then $q_{m,k}\geq q_{2,n-2}\geq q_{2,6}>2$. Hence, \eqref{cond} is satisfied if
and only if $n\leq 7$.

\vspace{3mm}

b) We assume that $\Omega$ is convex and that $n\geq8$.
Note that $q_{\frac n2,\frac n2}=\frac{n}{2+\sqrt{\frac n2-1}}<\frac n2$, and thus
\[p_{m,k}>2+\frac{4}{\frac n2 -2}=\frac{2n}{n-4}.\]
Hence, without loss of generality we may assume that
\[\frac{2n}{n-4}\leq p<p_{m,k}\]
and we can choose nonnegative numbers $\alpha$ and
$\beta$ such that $\alpha^2<m-1$, $\beta^2<k-1$, and
\begin{equation}\label{creix}p=2+\frac{4}{\frac{m}{2+\alpha}+\frac{k}{2+\beta}-2}.\end{equation}
This is because the
expression \eqref{creix} is increasing in $\alpha$ and $\beta$, and its value for
$\alpha=\beta=0$ is $\frac{2n}{n-4}$.
In addition, since $q_{m,k}\geq q_{2,n-2}\geq q_{2,6}>2$, we have
that $\frac{m}{2+\alpha}+\frac{k}{2+\beta}-2>0$ and that one of the
numbers $\frac{m}{2+\alpha}-1$ or $\frac{k}{2+\beta}-1$ is positive.

Hence, we can apply now Proposition \ref{sobx^A} to $u=u(\sigma,\tau)$ with
$a=\frac{m}{2+\alpha}-1$, $b=\frac{k}{2+\beta}-1$ and $q=2<D=\frac{m}{2+\alpha}+\frac{k}{2+\beta}$. 
We deduce that
\[\left(\int_{\widetilde U}\sigma ^{\frac{m}{2+\alpha}-1}\tau^{\frac{k}{2+\beta}-1}|u|^p
d\sigma d\tau\right)^{1/p}\leq C\left(\int_{\widetilde U}\sigma ^{\frac{m}{2+\alpha}-1}
\tau^{\frac{k}{2+\beta}-1}|\nabla_{(\sigma,\tau)}u|^2
d\sigma d\tau\right)^{1/2}.\]
Here we have extended $u$ by zero outside $\widetilde U$, obtaining a nonnegative Lipschitz function.
By Remark \ref{symm} it satisfies $u_s<0$ and $u_t<0$ whenever $u>0$, $s>0$, and $t>0$ since $\Omega$ is convex,
and therefore $u_{\sigma}<0$ and $u_{\tau}<0$ whenever $u>0$, $\sigma>0$, and $\tau>0$.
Note also that $q^*=2^*=\frac{2D}{D-2}=2+\frac{4}{D-2}=p$.
Thus, combining the last inequality with \eqref{sbarra}, we have
\[\left(\int_{\widetilde U}\sigma ^{\frac{m}{2+\alpha}-1}\tau^{\frac{k}{2+\beta}-1}|u|^p
d\sigma d\tau\right)^{1/p}\leq C \left(\|u\|_{L^{\infty}(\Omega_{\delta})}+
\|f(u)\|_{L^{\infty}(\Omega_{\delta})}\right).\]

Finally, since
\[\int_{\widetilde U}\sigma ^{\frac{m}{2+\alpha}-1}\tau^{\frac{k}{2+\beta}-1}|u|^pd\sigma d\tau=
c_{\alpha,\beta}\int_{\widetilde \Omega}s^{m-1}t^{k-1}|u|^pdsdt=c_{\alpha,\beta,m,k}\|u\|_{L^p(\Omega)}^p,\]
we conclude 
\[\|u\|_{L^p(\Omega)}\leq C\left(\|u\|_{L^{\infty}(\Omega_{\delta})}
+\|f(u)\|_{L^{\infty}(\Omega_{\delta})}\right).\]

\vspace{3mm}

c) Here we do not assume $\Omega$ to be convex. We set $\alpha=0$ in Lemma \ref{lema}. 
Estimate \eqref{joquese} in its proof gives
\[\int_{\Omega\backslash \Omega_{\delta/2}}u_s^2s^{-2}dx\leq
C\int_{\Omega_{\delta/2}}u_s^2dx,\]
and therefore, for a different constant $C$,
\[\int_{\Omega}u_s^2dx\leq C\int_{\Omega_{\delta/2}}u_s^2dx.\]
Since, for $1\leq i\leq m$ and $m+1\leq j\leq n$, $u_{x_i}=u_s\frac{x_i}{s}$ and $u_{x_j}=u_t\frac{x_j}{t}$, 
this leads to
\[\|u\|_{H^1_0(\Omega)}\leq C\|\nabla u\|_{L^2(\Omega)}\leq C\|u\|_{H^1(\Omega_\delta)},\]
as claimed.\qed

\section{Regularity of the extremal solution}\label{extremal}

This section is devoted to give the proof of Theorem \ref{th1}.
The estimates for convex domains will follow easily from Proposition \ref{prop1} and
the boundary estimates in convex domains of de Figueiredo, Lions, and Nussbaum \cite{DLN}.
These boundary estimates (see also \cite{C4} for their proof) follow easily from the 
moving planes method \cite{GNN}.

\begin{thm}[\cite{DLN},\cite{GNN}]\label{boundary} Let $\Omega$ be a smooth, bounded,
and convex domain, $f$ be any Lipschitz function, and $u$ be a bounded positive solution
of \eqref{pbl}. Then, there exist constants $\delta>0$ and $C$, both depending only on
$\Omega$, such that \[\|u\|_{L^{\infty}(\Omega_{\delta})}\leq C\|u\|_{L^1(\Omega)},\]
where $\Omega_{\delta}=\{x\in\Omega: {\rm dist}(x,\partial\Omega)<\delta\}$.
\end{thm}

We can now give the proof of Theorem \ref{th1}. The main part of the proof are the estimates
for non-convex domains. They will be proved by interpolating the $W^{1,p}$ and $W^{2,p}$
estimates of Nedev \cite{N} and our estimate of Lemma \ref{lema}, and by applying the
classical Sobolev inequality as explained in Remark \ref{remsob}.

\vspace{3mm}

\noindent \emph{Proof of Theorem \ref{th1}.} As we have pointed out, the estimates for
convex domains are a consequence of Proposition \ref{prop1} and Theorem \ref{boundary}.
Namely, we can apply the estimates of Proposition \ref{prop1} to the bounded and semi-stable minimal solutions
$u_\lambda$ of \eqref{pb} for $\lambda<\lambda^*$, and then by monotone convergence
the estimates hold for the extremal solution $u^*$. Note that $\|u_\lambda\|_{L^1(\Omega)}
\leq \|u^*\|_{L^1(\Omega)}<\infty$ for all $\lambda<\lambda^*$.

To prove part c) for convex domains, we use part c) of Proposition \ref{prop1} with $\delta$ 
replaced by $\delta/2$ and $\delta$ given by Theorem \ref{boundary}. We then control 
$\|u\|_{H^1(\Omega_{\delta/2})}$ by $\|u\|_{L^{\infty}(\Omega_\delta)}+\|f(u)\|_{L^{\infty}(\Omega_\delta)}$ 
using boundary estimates. Finally, we use Theorem \ref{boundary}.
Next we prove the estimates in parts a) and c) for non-convex domains.

We start by proving part a) when $\Omega$ is not convex. We have that $n=4$, i.e. $m=k=2$.
In \cite{N} (see its Remark 1) it is proved that the extremal solution satisfies
$u^*\in W^{1,p}(\Omega)$ for all $p<\frac{n}{n-3}$. Thus, since $n=4$, for each $p<4$ we have
\[\int_{\Omega}|u^*_s|^{p}dx\leq C\qquad \textrm{and} \qquad \int_{\Omega}|u^*_t|^pdx\leq C.\]

Assume that $\|u^*\|_{L^\infty(\Omega_\delta)}\leq C$ for some $\delta>0$ ---which we will prove later. Then,
by Lemma~\ref{lema}, for all $\gamma<4$ we have
\[\int_{\Omega}s^{-\gamma}|u_s^*|^2dx\leq C\qquad \textrm{and}\qquad \int_\Omega
t^{-\gamma}|u^*_t|^2dx\leq C.\]
Hence, for each $\lambda\in[0,1]$,
\[\int_\Omega (s^{-\lambda\gamma}|u^*_s|^{p-\lambda(p-2)}+t^{-\lambda\gamma}
|u^*_t|^{p-\lambda(p-2)})dx\leq C.\]
Setting now $\sigma=s^\kappa$, $\tau=t^\kappa$, and 
\[\kappa=1+\frac{\lambda\gamma}{p-\lambda(p-2)},\]
we obtain
\[\int_{\widetilde U} \sigma^{\frac{2}{\kappa}-1}\tau^{\frac{2}{\kappa}-1}|\nabla_{(\sigma,\tau)}
u^*|^{p-\lambda(p-2)}d\sigma d\tau\leq C,\]
and taking $p=3$, $\gamma=3$ and $\lambda=3/4$ (and thus $\kappa=2$), we obtain 
\[\int_{\widetilde U} |\nabla_{(\sigma,\tau)}
u^*|^{9/4}d\sigma d\tau\leq C.\]
Finally, applying Sobolev's inequality in the 2 dimensional plane $(\sigma,\tau)$,
$u^*\in L^\infty(\Omega)$.

It remains to prove that $\|u^*\|_{L^\infty(\Omega_\delta)}\leq C$ for some $\delta>0$. Since $u^*\in W^{1,p}(\Omega)$
for every $p<4$, we have
\[\int_{\Omega_\delta}st|\nabla u^*|^pdsdt\leq C.\]
Since the domain is smooth, we must have $0\notin\partial\Omega$ (otherwise the boundary
would have an isolated point) and hence, there exist $r_0>0$ and $\delta>0$ such that
$\Omega_\delta \cap B_{r_0}(0)=\emptyset$.
Thus, $s\geq r_0/\sqrt2$ in $\Omega_\delta\cap \{s>t\}$ and $t\geq r_0/\sqrt2$ in 
$\Omega_\delta\cap \{s<t\}$. It follows that
\[\int_{\Omega_\delta\cap \{s>t\}}t|\nabla u^*|^pdsdt\leq C\qquad \textrm{and}\qquad
\int_{\Omega_\delta\cap \{s<t\}}s|
\nabla u^*|^pdsdt\leq C.\]
Taking $p\in(3,4)$, we can apply Sobolev's inequality in dimension 3 (as explained in Remark \ref{remsob}),
to obtain $u^*\in L^\infty(\Omega_\delta \cap\{s>t\})$ and $u^*\in L^\infty(\Omega_\delta 
\cap\{s<t\})$.
Note that $u^*$ does not vanish through all $\partial(\Omega_\delta\cap\{s>t\})$ and 
$\partial(\Omega_\delta\cap\{s<t\})$,
but it vanishes on their intersection with 
$\partial\Omega$ ---a sufficiently large part of $\partial(\Omega_\delta\cap\{s>t\})$ 
and $\partial(\Omega_\delta\cap\{s<t\})$
to apply the Sobolev inequality.
Therefore $u^*\in L^\infty(\Omega_\delta)$, as claimed.

\vspace{3mm}

To prove part c) in the non-convex case, let $n\leq 6$. By Proposition \ref{prop1}, it suffices
to prove that $u^*\in H^1(\Omega_\delta)$ for some $\delta>0$. Take $r_0$ and $\delta$ such that
$\Omega_\delta \cap B_{r_0}(0)=\emptyset$, as in part a).

In \cite{N} it is proved that $u^*\in W^{2,p}(\Omega)$ for $p<\frac{n}{n-2}$. Thus, 
by the previous lower bounds for $s$ and $t$
in $\{s>t\}$ and $\{s<t\}$ respectively,
\[\int_{\Omega_\delta\cap \{s>t\}}t^{k-1}|D^2 u^*|^pdsdt\leq C\qquad\textrm{and}\qquad \int_{\Omega_\delta\cap
\{s<t\}}s^{m-1}|D^2 u^*|^pdsdt\leq C.\]
Since $n\leq 6$, $m\geq2$, and $k\geq2$, we have that $k\leq4$ and $m\leq 4$. It follows that
$\frac{2k+2}{k+3}<\frac{n}{n-2}$ and $\frac{2m+2}{m+3}<\frac{n}{n-2}$. 
Thus, we may take $p=\frac{2k+2}{k+3}$ and $p=\frac{2m+2}{m+3}$ respectively in the two previous estimates.
Now applying Sobolev's
inequality in dimension $k+1$ and $m+1$ respectively, we obtain
$\nabla u^*\in L^2(\Omega_\delta \cap\{s>t\})$ and $\nabla u^*\in L^2(\Omega_\delta \cap\{s<t\})$. Therefore,
$u^*\in H^1(\Omega_\delta)$.\qed

\section{Weighted Sobolev inequality}\label{sobolev}

It is well known that the classical Sobolev inequality can be deduced from the isoperimetric inequality.
This is done by applying first the isoperimetric inequality to the level sets of the function and
then using the coarea formula. In this way one deduces the Sobolev inequality with exponent 1 on
the gradient. Then, by applying
H\"older's inequality one deduces the general Sobolev inequality.
Here, we will proceed in this way to prove the Sobolev inequality of Proposition \ref{sobx^A}.

Recall that we will apply this Sobolev inequality to the function $u$ defined on the $(\sigma,\tau)$-plane, 
where $\sigma=s^{2+\alpha}$ and $\tau=t^{2+\beta}$.
Recall also that this application will be in convex domains, and thus $u$ satisfies the hypothesis of 
Proposition \ref{sobx^A}, i.e., $u_\sigma\leq0$ and $u_\tau\leq0$, with strict inequality whenever $u>0$. 
Hence, since the isoperimetric inequality will be applied to the level sets of $u$, it suffices
to prove a weighted isoperimetric inequality for bounded domains $\widetilde U\subset(\mathbb R_+)^2=(0,\infty)^2$
satisfying the following property:
\begin{itemize}
\item[(P)] For all $(\sigma,\tau)\in\widetilde U$, $\widetilde U(\cdot,\tau):=\{\sigma'>0:
(\sigma',\tau)\in\widetilde U\}$ and $\widetilde U(\sigma,\cdot):=\{\tau'>0: (\sigma,\tau')\in\widetilde U\}$ 
are intervals which are strictly decreasing in $\tau$ and $\sigma$, respectively.
\end{itemize}

We denote
\[m(\widetilde U)=\int_{\widetilde U} \sigma^a\tau^bd\sigma d\tau\qquad \mbox{and}\qquad m(\partial\widetilde 
U\cap (\mathbb R_+)^2)=
\int_{\partial\widetilde U\cap(\mathbb R_+)^2}\sigma^a\tau^bd\sigma d\tau.\]
Note that in the weighted perimeter $m(\partial \widetilde U \cap (\mathbb R_+)^2)$ the part of 
$\partial\widetilde U$ on the $\sigma$ and $\tau$ coordinate axes is not counted.
The following isoperimetric inequality holds in domains satisfying property (P) above, under no further 
regularity assumption on them.

\begin{prop}\label{isop} Let $\widetilde U\subset (\mathbb R_+)^2$ be a bounded domain
satisfying (P) above, $a>-1$ and $b>-1$ be real numbers, being positive at least
one of them, and \[D=a+b+2.\]

Then, there exists a constant $C$ depending only on $a$ and $b$ such that
\[m(\widetilde U)^{\frac{D-1}{D}}\leq Cm(\partial\widetilde U\cap(\mathbb R_+)^2).\]
\end{prop}

\begin{proof} First, by symmetry we can suppose $a>0$.

Property (P) ensures that there exists a unique well defined decreasing, bounded, and continuous function
$\psi:(0,\overline\sigma)\rightarrow (0,\infty)$ for some $\overline\sigma>0$ such that
\begin{equation}\label{fcngraph}\widetilde U=\{(\sigma,\tau)\in(\mathbb R_+)^2: \tau<\psi(\sigma)\}.\end{equation}
In addition, extending $\psi$ by zero in $[\overline\sigma,\infty)$, $\psi$ is continuous and nonincreasing.
Even that we could have $\psi'=-\infty$ at some points, $|\psi'|=-\psi'$ is integrable (since $\psi$ is bounded) 
and thus $\psi\in W^{1,1}(\mathbb R)$. We have that
\[m(\widetilde U)=\frac{1}{b+1}\int_0^{+\infty} \sigma^a\psi^{b+1}d\sigma\ \textrm{ and }\ m(\partial
\widetilde U\cap(\mathbb R_+)^2)=\int_0^{+\infty}\sigma^a\psi^b
\sqrt{1+\psi'^2}d\sigma.\]

Let $\mu>0$ be such that 
\begin{equation}\label{valmu}m(\widetilde U)=\frac{\mu^D}{(a+1)(b+1)}.\end{equation}
We claim that
\[\psi(\sigma)<\mu\ \textrm{ for }\ \sigma>\mu.\]
Assume that this is false. Then, we would have $\psi(\sigma')\geq\mu$
for some $\sigma'>\mu$, and hence
\[m(\widetilde U)\geq\frac{1}{b+1}\int_0^{\sigma'}\sigma^a\psi^{b+1}d\sigma>\frac{1}{b+1}\int_0^\mu
\sigma^a\mu^{b+1}d\sigma=\frac{\mu^D}{(a+1)(b+1)},\]
a contradiction.
On the other hand, since $a>0$, $b+1>0$, and $\psi'\leq0$,
\begin{eqnarray*} m(\partial\widetilde U\cap(\mathbb R_+)^2)&=&\int_0^{+\infty}\sigma^a\psi^b\sqrt{1+\psi'^2}d\sigma\\
&\geq& c\int_0^{+\infty}\sigma^a\psi^b\left\{1-\frac{b+1}{a}\psi'\right\}d\sigma \\
& = & c\int_0^{+\infty}\sigma^a\left\{\psi^b-\frac{d}{d\sigma}\left(\frac{\psi^{b+1}}{a}\right)\right\}d\sigma\\
& = & c\int_0^{+\infty}\sigma^a\psi^{b+1}\left(\frac{1}{\psi}+\frac{1}{\sigma}\right)d\sigma,\end{eqnarray*}
for some constant $c$ depending only on $a$ and $b$.

Finally, taking into account that $\psi(\sigma)<\mu$ for $\sigma>\mu$, we obtain that $\frac{1}{\psi}
+\frac{1}{\sigma}\geq \frac{1}{\mu}$ for each $\sigma>0$. Thus, recalling \eqref{valmu},
\[m(\partial\widetilde U\cap(\mathbb R_+)^2)\geq c\int_0^{+\infty}\sigma^a\psi^{b+1}\left(\frac{1}{\psi}+
\frac{1}{\sigma}\right)d\sigma
\geq \frac{c}{\mu} m(\widetilde U)=cm(\widetilde U)^{\frac{D-1}{D}},\] as claimed.
\end{proof}

Now we are able to prove our Sobolev inequality from the previous isoperimetric inequality. We follow the 
proof given in \cite{D} for the classical unweighted case.

\vspace{3mm}

\noindent \emph{Proof of Proposition \ref{sobx^A}.} We will prove first the case $q=1$.

Letting $\chi_A$ denote the characteristic function of the set $A$, we have
\[u(\sigma,\tau)=\int_0^{+\infty}
\chi_{[u(\sigma,\tau)>\lambda]}d\lambda.\]
Thus, by Minkowski's integral inequality
\begin{eqnarray*}\left(\int_{(\mathbb R_+)^2}\sigma^a\tau^b|u|^{\frac{D}{D-1}}d\sigma d\tau\right)^{\frac{D-1}{D}}
&\leq& \int_0^{+\infty}\left(\int_{(\mathbb R_+)^2}\sigma^a\tau^b
\chi_{[u(\sigma,\tau)>\lambda]}d\sigma d\tau\right)^{\frac{D-1}{D}}d\lambda\\ &=&\int_0^{+\infty}
m(\{u(\sigma,\tau)>\lambda\})^{\frac{D-1}{D}}d\lambda.\end{eqnarray*}
Since $u_\sigma\leq0$ and $u_\tau\leq0$, with strict inequality when $u>0$, the level sets
$\{u(\sigma,\tau)>\lambda\}$ satisfy property (P) in the beginning of Section 4. In fact, since $u_{\tau}<0$ 
at points where $u=\lambda>0$, the implicit function theorem gives that the function $\psi$ in \eqref{fcngraph} 
when $\widetilde U=\{u(\sigma,\tau)>\lambda\}$ is $C^1$ in $(0,\overline\sigma)$. Thus,
Proposition \ref{isop} leads to
\begin{eqnarray*} m\left(\{u(\sigma,\tau)>\lambda\}\right)^{\frac{D-1}{D}}&\leq& Cm\left(\partial\{u(\sigma,\tau)
>\lambda\}\cap (\mathbb R_+)^2\right)\\
&=& Cm\left(\{u(\sigma,\tau)=\lambda\}\cap(\mathbb R_+)^2\right),\end{eqnarray*}
whence
\[\left(\int_{(\mathbb R_+)^2}\sigma^a\tau^b|u|^{\frac{D}{D-1}}d\sigma d\tau\right)^{\frac{D-1}{D}}
\leq C\int_0^{+\infty}m\left(\{u(\sigma,\tau)=\lambda\}\cap(\mathbb R_+)^2\right)d\lambda.\]

Let $u_{ev}$ be the even extension of $u$ with respect to $\sigma$ and $\tau$ in $\mathbb R^2$. Then,
\[\int_0^{+\infty}m\left(\{u(\sigma,\tau)=\lambda\}\cap(\mathbb R_+)^2\right)d\lambda=\frac14
\int_0^{+\infty}m\left(\{u_{ev}(\sigma,\tau)=\lambda\}\right)d\lambda,\]
and by the coarea formula
\[\int_0^{+\infty}m\left(\{u_{ev}(\sigma,\tau)=\lambda\}\right)d\lambda
=\int_{\mathbb R^2}\sigma^a\tau^b|\nabla u_{ev}|d\sigma d\tau.\]
Thus, we obtain
\[\left(\int_{(\mathbb R_+)^2}\sigma^a\tau^b|u|^{\frac{D}{D-1}}d\sigma d\tau\right)^{\frac{D-1}{D}}\leq C 
\int_{(\mathbb R_+)^2}\sigma^a\tau^b|\nabla u|d\sigma d\tau,\]
and the proposition is proved for $q=1$.

Finally, let us prove the case $1<q<D$. Take $u$ satisfying the hypotheses of Proposition
\ref{sobx^A}, and define $v=u^{\gamma}$, where $\gamma=\frac{q^*}{1^*}$. Since $\gamma>1$,
we have that $v$ also satisfies the hypotheses of the proposition, and we can apply the weighted Sobolev 
inequality with $q=1$ to get
\begin{eqnarray*}\left(\int_{(\mathbb R_+)^2}\sigma^a\tau^b|u|^{q^*}d\sigma d\tau\right)^{1/1^*}&=&
\left(\int_{(\mathbb R_+)^2}\sigma^a\tau^b|v|^{\frac{D}{D-1}}d\sigma d\tau\right)^{\frac{D-1}{D}}\\
&\leq& C\int_{(\mathbb R_+)^2}\sigma^a\tau^b|\nabla v|d\sigma d\tau.\end{eqnarray*} 
Now, $|\nabla v|=\gamma u^{\gamma-1}|\nabla u|$,
and by H\"older's inequality it follows that
\[\int_{(\mathbb R_+)^2}\sigma^a\tau^b|\nabla v|d\sigma d\tau\leq C\left(\int_{(\mathbb R_+)^2}\sigma^a\tau^b
|\nabla u|^qd\sigma d\tau\right)^{1/q}\left(\int_{(\mathbb R_+)^2}\sigma^a\tau^b|u|^{(\gamma-1)q'}
d\sigma d\tau\right)^{1/q'}.\]
But from the definition of $\gamma$ and $q^*$ it follows that 
\[\frac{\gamma-1}{q^*}=\frac{1}{1^*}-\frac{1}{q^*}=
\frac{1}{q'},\qquad (\gamma-1)q'=q^*,\] and hence
\[\left(\int_{(\mathbb R_+)^2}\sigma^a\tau^b|u|^{q^*}d\sigma d\tau\right)^{1/q^*}\leq
C\left(\int_{(\mathbb R_+)^2}\sigma^a\tau^b|\nabla u|^{q}d\sigma d\tau\right)^{1/q},\] as desired.
\qed


\begin{thebibliography}{99}

\bibitem{BV} H. Brezis, J.L. V\'azquez, \emph{Blow-up solutions of some nonlinear elliptic problems},
Rev. Mat. Univ. Compl. Madrid 10 (1997), 443-469.

\bibitem{C4} X. Cabr\'{e}, \emph{Regularity of minimizers of semilinear elliptic problems up to
dimension four}, Comm. Pure Appl. Math. 63 (2010), 1362-1380.

\bibitem{CC} X. Cabr\'{e}, A. Capella, \emph{Regularity of radial minimizers and extremal solutions
of semi-linear elliptic equations}, J. Funct. Anal. 238 (2006), 709-733.

\bibitem{CC2} X. Cabr\'{e}, A. Capella, \emph{Regularity of minimizers for three elliptic problems:
minimal cones, harmonic maps, and semilinear equations}, Pure and Applied Math Quarterly 3 (2007), 801-825.

\bibitem{CR} X. Cabr\'e, X. Ros-Oton, \emph{Sobolev and isoperimetric inequalities with monomial weights},
preprint.

\bibitem{CS} X. Cabr\'e, M. Sanch\'on, \emph{Geometric-type Sobolev inequalities and applications to 
the regularity of minimizers},
arXiv:1111.2801v1.

\bibitem{DLN} D. de Figueiredo, P.L. Lions, R.D. Nussbaum, \emph{A priori estimates and existence of
positive solutions of semilinear elliptic equations}, J. Math. Pures Appl. 61 (1982), 41-63.

\bibitem{D} L. Dupaigne, \emph{Stable Solutions to Elliptic Partial Differential Equations},
CRC Press, 2011.

\bibitem{GNN} B. Gidas, W.M. Ni, L. Nirenberg, \emph{Symmetry and related properties via the maximum
principle}, Comm. Math. Phys. 68 (1979), 209-243.

\bibitem{H} P. Hajlasz, \emph{Sobolev spaces on an arbitrary metric space}, Potential Analysis 5 (1996), 403-415.

\bibitem{Mo} F. Morgan, \emph{Manifolds with Density}, Notices of the American Mathematical Society 52 (2005), 853-858.

\bibitem{N} G. Nedev, \emph{Regularity of the extremal solution of semilinear elliptic equations},
C. R. Acad. Sci. Paris S\'{e}r. I Math. 330 (2000), 997-1002.

\bibitem{N2} G. Nedev, \emph{Extremal solutions of semilinear elliptic equations}, preprint, 2001.


\end{thebibliography}
\end{document}